\documentclass[11pt, twoside]{amsart}

\title{Algebras of Quasi-Pl\"ucker Coordinates are Koszul}
\date{\today}
\author{Robert Laugwitz}
\address{Department of Mathematics,
Rutgers University,
110 Frelinghuysen Road,
Piscataway, NJ 08854-8019, USA}
\email{robert.laugwitz@rutgers.edu}
\urladdr{https://www.math.rutgers.edu/~rul2/}
\author{Vladimir Retakh}
\address{Department of Mathematics,
Rutgers University,
110 Frelinghuysen Road,
Piscataway, NJ 08854-8019, USA}
\email{vretakh@math.rutgers.edu}
\urladdr{https://www.math.rutgers.edu/~vretakh/}

\usepackage{mathabx}
\usepackage{amsmath}
\usepackage{mathrsfs}
\usepackage{amsthm}
\usepackage{amssymb}
\usepackage[alphabetic, initials]{amsrefs}
\usepackage[english]{babel}
\usepackage{url}
\usepackage{fancyhdr}
\usepackage{graphicx}
\usepackage{color}
\usepackage[
colorlinks=true,
linkcolor=black, 
anchorcolor=black,
citecolor=black, 
urlcolor=black, 
]{hyperref}
\usepackage{tikz}

\newcommand{\der}{\operatorname{d}}
\newcommand{\gr}{\operatorname{gr}}
\newcommand{\BPl}[1]{B_n^{(#1)}}
\newcommand{\CPl}[1]{C_n^{(#1)}}
\newcommand{\QPl}[1]{Q_n^{(#1)}}
\newcommand{\RPl}[1]{R_n^{(#1)}}
\newcommand{\FPl}[1]{F_{#1}}
\newcommand{\op}[1]{\operatorname{#1}}
\newcommand{\cO}{\mathcal{O}}
\newcommand{\mC}{\mathbb{C}}

\newtheorem{theorem}{Theorem}[]
\newtheorem{proposition}[theorem]{Proposition}
\newtheorem{corollary}[theorem]{Corollary}
\newtheorem{lemma}[theorem]{Lemma}
\newtheorem{theorem*}{Theorem}

\theoremstyle{definition}
\newtheorem{definition}[theorem]{Definition}

\theoremstyle{remark}
\newtheorem{example}[theorem]{Example}


\begin{document}

\begin{abstract}
Motivated by the theory of quasi-determinants, we study non-commutative algebras of quasi-Pl\"ucker coordinates. We prove that these algebras provide new examples of non-homogeneous quadratic Koszul algebras by showing that their quadratic duals have quadratic Gr\"obner bases.
\end{abstract}

\maketitle


\section{Introduction}

The Koszul property  of the commutative quadratic algebra of Pl\"ucker coordinates is a well-known fact (see \cite{MS}*{Theorem 14.6} for a textbook exposition). In this paper we introduce and study non-commutative analogues of this algebra, using the quasi-Pl\"ucker coordinates defined in \cite{GR4}*{Section~II}. In particular, we establish the Koszul property for these non-homogeneous quadratic algebras.

We denote by $\underline{n}$ the set $\lbrace 1,2,\ldots,n\rbrace$. Given ordered sets $I\subset J$, we denote by $J\setminus I$ the ordered set obtained by removing $I$, and by $J|K$ the ordered set obtained by appending an ordered set $K$. The set $\lbrace j\rbrace$ containing one element $j$ is simply denoted by $j$.

\subsection{Commutative Pl\"ucker Coordinates}\label{classicalsection}

For $k\leq n$ and a $k\times n$-matrix $A$ with commutative entries we can choose a subset $I=\lbrace i_1,\ldots,i_k\rbrace$ of the column indices $\underline{n}$ and consider the \emph{Pl\"ucker coordinates}
\begin{equation}
p_I(A):=\det A(i_1,\ldots, i_k),
\end{equation}
using the $k\times k$ submatrix with columns corresponding to the indices in $I$. It is well-known (see e.g. \cite{HP}*{Chapter VII.6}) that the $p_I(A)$ satisfy $\op{GL}_n$-invariance, skew-symmetry with respect to commuting columns, and the \emph{Pl\"ucker identity}
\begin{equation}\label{classicalpluecker}
\sum_{t=1}^{k+1}(-1)^t p_{I|j_t}(A)p_{J\setminus j_t }(A) = 0,
\end{equation}
for subsets $I=\lbrace i_1,\ldots, i_{k-1}\rbrace$ and $J=\lbrace j_1,\ldots,j_{k+1}\rbrace$ of the column indices. The ideal of all relations among the Pl\"ucker coordinates is generated by these relations (\cite{Wei}*{Chapter IV \S 5}, cf. \cite{Stu}*{Section 3.1} or \cite{DK}*{Theorem 4.4.5}).

For example, let $k=2<3=n$. Then we can study a commutative algebra generated by elements $p_{12}, p_{13}, p_{23}$, and the Pl\"ucker relations add no extra relations.
Letting $k=2<4=n$ and choosing $I=\lbrace 1\rbrace$, $J=\lbrace 2,3,4\rbrace$ one obtains the classical identity
\begin{equation}\label{smallex}
p_{12}p_{34}-p_{13}p_{24}+p_{14}p_{23}=0.
\end{equation}

For $k=3$ and $n=6$ we, for example, get the relations
\begin{align}
p_{123}p_{456}-p_{124}p_{356}+p_{125}p_{346}-p_{126}p_{345}&=0,\label{k3eqpl1}\\
p_{123}p_{245}-p_{124}p_{235}+p_{125}p_{234}&=0,\label{k3eqpl2}
\end{align}
where $I=\lbrace 1,2\rbrace$ and $J=\lbrace 3,4,5,6\rbrace$ in Equation (\ref{k3eqpl1}) and $J=\lbrace 2,3,4,5\rbrace$ in Equation (\ref{k3eqpl2}),
plus similar relations interchanging the roles of the numbers in $\underline{6}$.

One can consider the symbols $p_I$ as generators of a quadratic commutative algebra, the quadratic quotient  $\cO_{k,n}$ of the polynomial algebra $\mC[p_{I}\mid I\subset \underline{n}]$ by the relations (\ref{classicalpluecker}), and skew-symmetry with respect to commuting indices. It is well-known that $\cO_{k,n}$  is a Koszul ring since the relations give a quadratic Gr\"obner basis. This was proved in \cite{GRS}, and also follows from \cite{Kem} (where Koszul rings are called \emph{wonderful rings}), using results of \cite{DEP}. The result of \cite{GRS} was reinterpreted in Gr\"obner basis terminology in \cite{SW}, see also \cite{MS}*{Theorem 14.6} for a textbook exposition.

The Hilbert series of $\cO_{k,n}$ can be computed combinatorially using methods from \cite{Stu}. 
In the above example of $\cO_{2,4}$, one obtains the closed formula (see \cite{GW}*{Section~7})
\begin{equation}
H(\cO_{2,4},t)=\frac{1+t}{(1-t)^5}=1 + 6 q + 20 q^2 + 50 q^3 + 105 q^4 + 196 q^5 + O(q^6).
\end{equation}

The Pl\"ucker coordinates $p_I(A)$ define an embedding of the Grassmannian $G_{k,n}$ into projective space of dimension $\binom{n}{k}-1$. The coordinate ring of $G_{k,n}$ via the Pl\"ucker embedding is the quadratic algebra $\cO_{k,n}$ considered above.

\subsection{Non-commutative Pl\"ucker Coordinates}\label{quasidets}

Analogues of Pl\"ucker coordinates for a $k\times n$-matrix with non-commuting entries are obtained using the theory of quasi-de\-ter\-mi\-nants \cites{GR4,GGRW} as ratios of two quasi-minors. More precisely, given a choice of two indices $i,j\in \underline{n}$, and a subset $i\notin I \subset \underline{n}$ of size $k-1$ and a matrix $A$ with coefficients in a division ring, the \emph{quasi-Pl\"ucker coordinate} $q_{ij}^I$ is defined as the following ratio of non-commutative analogues of maximal minors:
\begin{equation}\label{quasidetformula}
q_{ij}^I=q_{ij}^I(A)=\underbrace{\begin{vmatrix}
a_{1i}&a_{1i_1}&\hdots& a_{1i_{k-1}}\\
\vdots && \vdots \\
a_{ki}&a_{ki_1}&\hdots& a_{ki_{k-1}}
\end{vmatrix}_{si}^{-1}}_{p^{(s)}_{i| I}(A)^{-1}}
\underbrace{\begin{vmatrix}
a_{1j}&a_{1i_1}&\hdots& a_{1i_{k-1}}\\
\vdots && \vdots \\
a_{kj}&a_{ki_1}&\hdots& a_{ki_{k-1}}
\end{vmatrix}_{sj}}_{p^{(s)}_{j| I}(A)},
\end{equation}
which is independent of choice in $s$, undefined if $i\in I$, and zero if $j\in I$.
The following analogue of the Pl\"ucker relations holds for these  non-commutative analogues of Pl\"ucker coordinates:
\begin{equation}\label{quasiplueckerel}
\sum_{j\in L} q_{ij}^Mq_{ji}^{L\setminus j }=1.
\end{equation}
In the case where the entries of $A$ commute, (\ref{quasiplueckerel}) recovers the classical relation (\ref{classicalpluecker}).
Moreover, symmetry in changing the order of elements of $I$ holds, replacing skew-symmetry for these ratios, and $q_{ji}^I$ is inverse to $q_{ij}^I$ if non-zero. By considering the ratios $q_{ij}$, an additional relation appears:
\begin{equation}
q_{ij}^{N\setminus \lbrace i,j\rbrace}q_{jm}^{N\setminus \lbrace j,m\rbrace}=-q_{im}^{N\setminus \lbrace i,m\rbrace}.
\end{equation}
See Section \ref{Rnk} for the list of relations among quasi-Pl\"ucker coordinates. 

For example, in the case $k=2$ and $n=4$, Equation (\ref{quasiplueckerel}) gives the formula
\begin{equation}
q_{13}^2q_{31}^4+q_{14}^2q_{41}^3=1.
\end{equation}
This translates to
\begin{equation}
p_{12}^{-1}p_{32}p_{34}^{-1}p_{14}+p_{12}^{-1}p_{42}p_{43}^{-1}p_{13}=1.
\end{equation}
If the elements $p_{ij}$ commute, this equality reduces to the classical formula (\ref{smallex}).

As a second example, consider the case $k=2$ and $n=6$. Let $M=\lbrace 1,2\rbrace$ and $L=\lbrace 3,4,5\rbrace$. Then, with $i=6$, we obtain the equation 
\begin{align}\begin{split}
q_{63}^{12}q_{36}^{45}+q_{64}^{12}q_{46}^{35}+q_{65}^{12}q_{56}^{34}&=1\\
\Longleftrightarrow\qquad p_{612}^{-1}p_{312}p_{345}^{-1}p_{645}+p_{612}^{-1}p_{412}p_{435}^{-1}p_{635}+p_{612}^{-1}p_{512}p_{543}^{-1}p_{643}&=1.
\end{split}
\end{align}
Assuming that the variables commute, this recovers relation (\ref{k3eqpl1}). Similarly, Equation (\ref{k3eqpl2}) can be recovered using $M=\lbrace 1,2\rbrace$, $L=\lbrace 2,3,4\rbrace$, $i=5$.

Note that it was shown in \cite{GR4}*{Theorem 2.1.6} that any $\op{GL}(n)$-invariant rational function over a free skew-field is a rational function of the quasi-Pl\"ucker coordinates. Moreover, \cite{BR}*{Proposition 2.41} shows that, for $k=2$, quasi-Pl\"ucker coordinates form a free skew-subfield within the free skew-field with $2n$ generators.

\subsection{Quantum Pl\"ucker Coordinates}\label{quasiplueckersect}

A quantum analogue of Equation (\ref{classicalpluecker}) was considered in \cite{TT}*{Eq. (3.2c)} in order to construct a quantum analogue of the coordinate algebra $\cO_{k,n}$ of the Grassmannian $G_{k,n}$. For this, more general exchange relations appear, called \emph{Young symmetry relations}:
\begin{equation}\label{quantumyoung}
\sum_{\Lambda\subseteq I, |\Lambda|=n}(-q)^{-l(I\setminus \Lambda|\Lambda)}f_{I\setminus \Lambda}f_{\Lambda|J},
\end{equation}
for $1\leq r\leq d$, and $I$, $J$ index sets of size $d+r$ and $d-r$, respectively. Here, we use notation adapted from \cite{Lau}*{Eq. (9)}. The classical Pl\"ucker relations (\ref{classicalpluecker}) can be recovered as the case $r=1$, $q=1$. It was further shown in \cite{Lau} that the relations (\ref{quantumyoung}) can be reduced successively to relations with $r=1$. 

In fact, the Young symmetry relations are consequences of the quasi-Pl\"ucker relations (\ref{quasiplueckerel}) \cite{Lau}*{Theorem 28}.

\subsection{Sagbi and Gr\"obner Basis for Coordinates of Grassmannians}

In the commutative setting, maximal minors form a \emph{sagbi basis} (canonical subalgebra basis) according to \cite{Stu}*{3.2.9}.
(The relations among these maximal minors give a quadratic Gr\"obner basis as mentioned in Section \ref{classicalsection}.)

This result was generalized to another approach to quantum Grassmannians, which emerges from geometry and quantum cohomology and is a commutative construction (rather than using $q$-commutators). The coordinate ring, denoted by $\Bbbk[\mathscr{C}_{k,n}^q]$ of the quantum Grassmannian consists of  maximal graded minors (of degree up to $q\in \mathbb{N}$) of $k\times n$-matrices with graded entries. In \cite{SS}*{Theorem 1} it is proved that these maximal graded minors give a sagbi basis for the coordinate ring $\Bbbk[\mathscr{C}_{k,n}^q]$ within the polynomial ring of graded entries of the matrix. Further, the relations among these maximal graded minors  have a quadratic Gr\"obner basis \cite{SS}*{Theorem 2}.

\subsection{This Paper's Approach}

This paper takes the approach to start with a quadratic algebra of quasi-Pl\"ucker coordinates (introduced in Section \ref{Rnk}). This algebra is quadratic-linear, and a theory for Koszulity of such algebras has been developed in \cite{PP}.

Our main result is that the associated quadratic algebra of this algebra has a quadratic Gr\"obner basis. Hence the algebra of non-commutative Pl\"ucker coordinates is a non-homogeneous Koszul algebra (Theorem \ref{RKoszulTheorem}). In Section \ref{limitalgebra} we consider colimits of these algebras, varying $k\geq 2$. We further study  a second version of algebras of quasi-Pl\"ucker coordinates which is not quadratic-linear, but also non-homogeneous Koszul in Section~\ref{limitalgebrakoszul}.

In Section \ref{diffsect} we study the Koszul dual dg algebras explicitly in the case $k=2$, and we finish the exposition by considering an algebra of non-commutative \emph{flag coordinates} which is also non-homogeneous Koszul in Section \ref{flagsect}.

There are different approaches to non-commutative Grassmannian coordinate rings, see e.g. \cites{Kap,KR}, which are not discussed here.

\section{Definition of the Algebra \texorpdfstring{$\RPl{k}$}{Rnk}}\label{Rnk}

We want to define a quadratic algebra of quasi-Pl\"ucker coordinates. As outlined in Section \ref{quasidets}, quasi-Pl\"ucker coordinates were constructed using quasi-determinants in \cite{GR4}*{Section~II}, cf. \cite{GGRW}*{4.3}. For fixed integers $n, k \geq 2$ we define the algebra $\QPl{k}$ as having generators $q_{ij}^I$, where $I\subset \underline{n}$ has size $k-1$ and $i\notin I,$ which satisfy the following relations, obtained from \cite{GGRW}*{4.3}:
\begin{enumerate}
\item[(i)] $q_{ij}^I$ does not depend on the ordering of the elements of $I$;
\item[(ii)] $q_{ij}^I=0$ whenever $j\in I$ and $i\neq j$;
\item[(iii)] $q_{ii}^I=1$, and $q_{ij}^Iq_{jl}^I=q_{il}^I$;
\item[(iv)] $q_{ij}^{N\setminus \lbrace i,j\rbrace}q_{jm}^{N\setminus \lbrace j,m\rbrace}=-q_{im}^{N\setminus \lbrace i,m\rbrace}$.
\item[(v)] If $i\notin M$, then $\sum_{j\in L} q_{ij}^Mq_{ji}^{L\setminus\lbrace j \rbrace}=1$.
\end{enumerate}
Relation (iv) is called \emph{non-commutative skew-symmetry}, and (v) is a non-com\-mu\-ta\-tive analogue of the Pl\"ucker relations.

The algebra $\QPl{2}$ is studied in \cite{BR}, as the algebra of \emph{non-commutative sectors}, where it is denoted by $\mathcal{Q}_n$. Given an $k\times n$-matrix $A$ with entries in a division ring, we note that the description of $q_{ij}^I$ in terms of quasi-determinants in Equation (\ref{quasidetformula}) provides a morphism of algebras from $\QPl{k}$ to the skew-field generated by the non-commutative entries $a_{ij}$ of the matrix $A$.

We also consider the subalgebra $\RPl{k}$ of $\QPl{k}$ generated by those of the $q_{ij}^I$ for which $i<j$. The restriction to $\RPl{k}$ can be justified by noting that the skew-fields generated by the images of $\RPl{k}$ and $\QPl{k}$ in the skew-field generated by the matrix entries $a_{ij}$ coincide. One advantage of considering $\RPl{k}$ is that it admits  a presentation as a quadratic-linear algebra:
\begin{proposition}
The subalgebra $\RPl{k}$ can be described by the relations
\begin{align}\label{relation1}
q_{ij}^I q_{jl}^I&=q_{il}^I, & \forall i<j<l, ~i,j\notin I,\\
\sum_{j=1}^{k-1} q_{l_0l_j}^Mq_{l_jl_{k}}^{L\setminus \lbrace l_j, l_k\rbrace}+q_{l_0l_k}^{L\setminus  \lbrace l_0, l_k\rbrace}&=q_{l_0l_k}^{M},\label{relation2}
\end{align}
where $L=\lbrace l_0<l_1<\dots <l_{k}\rbrace, l_0\notin M$, which are read in the way that $q_{ij}^M=0$ if $j\in M$.
Hence, $\RPl{k}$ is a quadratic-linear algebra.
\end{proposition}
\begin{proof}
Starting with formula (v), we distinguish three cases, depending on the value of the index $i$. In the case when $l_j<i<l_{j+1}$ for some $j=1, \ldots, k-1$ we obtain Equation (\ref{relation2}) by multiplying with $q_{l_1, i}^M$ on the left, and $q_{i,l_k}^{L\setminus  l_k}$ on the right. If $i<l_1$, it suffices to multiply by $q_{i,l_k}^{L\setminus  l_k}$ on the right; and if $l_k<i$, it is enough to multiply by $q_{l_1, i}^M$ on the left. In all three cases, we obtain the same relation after relabelling so that the index-set $L$ contains $i$ in the correct order. These are all possible relations between generators $q_{ij}^I$ with $i<j$ as in this case Equation (iv) is a special case of Equation (\ref{relation2}), with $M=L\setminus 
\lbrace l_0,l_j\rbrace$ for some $1\leq j\leq k$.
\end{proof}

\begin{example}
Let us consider the case $k=2$. In this case, the skew-symmetry relation (iv) and the Pl\"ucker relation (v) in $\QPl{2}$ become 
\begin{align}
q_{ij}^l q_{jl}^{i}&=-q_{il}^j\\
q_{ij}^m q_{ji}^l + q_{il}^mq_{li}^j&=1,
\end{align}
where all indices are distinct. In this case, the algebra $\RPl{2}$ has $(n-2)\left(\genfrac{}{}{0pt}{}{n}{2}\right)$ generators $q_{ij}^l$ with $i<j$ and $l\neq i,j$. The relations governing this algebra are
\begin{align}
q_{ij}^mq_{jl}^m-q_{il}^m&=0,\\
q_{ij}^mq_{jl}^i+q_{il}^j&=q_{il}^m.
\end{align}
for all $i<j<l$ and $m$ a distinct element from $i,j,l$, and  if $m=l$ we have the relation 
\begin{align}
q_{ij}^lq_{jl}^i+q_{il}^j&=0.
\end{align}
\end{example}


\section{Koszulness of the Algebra \texorpdfstring{$\RPl{k}$}{Rnk}}

The Koszul property for quadratic algebras can, more generally, be studied for non-ho\-mo\-gen\-e\-ous quadratic algebras \cite{PP}*{Chapter~5}. A non-ho\-mo\-gen\-e\-ous quadratic algebra is \emph{Koszul} if the corresponding quadratic algebra $A^{(0)}$ obtained by taking the homogeneous parts of the quadratic relations is Koszul. In this case, $A^{(0)}$ is isomorphic to the associated graded algebra $\gr A$.

We shall prove that such an algebra $A$ is Koszul by showing that the qua\-drat\-ic dual $\left(A^{(0)}\right)^!$ of $A^{(0)}$ is Koszul (cf. \cite{PP}*{Chapter 2, Corollary 3.3}). This, in turn, is proved by showing that $\left(A^{(0)}\right)^!$ has a  quadratic Gr\"obner basis of relations (giving a non-commutative PBW basis for the algebra) using the rewriting method, see e.g. \cite{LV}*{Theorem 4.1.1}.

The associated quadratic algebra $(\RPl{k})^{(0)}$ is generated by the relations 
\begin{align}
q_{ij}^I q_{jl}^I&=0, &\forall i<j<l,~i,j\notin I,\label{qrelation1}\\
\sum_{j=1}^{k-1} q_{l_0l_j}^M q_{l_jl_k}^{L\setminus \lbrace l_j, l_k\rbrace}&=0,\label{qrelation2}
\end{align}
where $L=\lbrace l_0<l_1<\dots <l_{k}\rbrace, l_0\notin M$, which are read in the way that $q_{ij}^M=0$ if $j\in M$.

We consider the quadratic dual of $(\RPl{k})^{(0)}$ which is denoted by $\BPl{k}$. It consists of generators $r_{ij}^I$ for $i<j$ and $i\notin I$, and $r_{ij}^I=0$ if $j\in I$. Again, we regard $I$ as a strictly ordered set of indices.

The following lemma follows by carefully constructing a basis for the orthogonal complement to the subspace of relations (\ref{qrelation1})--(\ref{qrelation2}) in degree two. We will write $i<K<j$ to denote that $i<l<j$ for every element $l\in K$.

\begin{lemma}\label{dualrelations}
The algebra $\BPl{k}$ is given by the relations
\begin{align}\label{dualrelation1}
r_{ij}^I&=0, &\text{ if $j\in I$},\\\label{dualrelation2}
r_{ij}^I r_{ab}^J&=0, &\text{unless }j=a \text{ and }\begin{cases} i\in J\text{ and }i< J\setminus i<b\\\text{or }I=J \end{cases},\\
r_{ij}^I r_{jl}^J&=r_{ij'}^I r_{j'l}^{J\setminus  j'\cup j}, &\forall j'\in J\setminus (J\cap I), \text{ provided }i<J\setminus i<l,\label{dualrelation3}
\end{align}
provided that $i\notin I$, $a\notin J$ for (\ref{dualrelation2}), and $i\notin I$, $j\notin J$ for (\ref{dualrelation3}).
\end{lemma}

\begin{theorem}\label{RKoszulTheorem}
The algebra $\BPl{k}$ has a quadratic non-commutative Gr\"obner basis, for $k,n\geq 2$, given by the relations (\ref{dualrelation2})--(\ref{dualrelation3}) on generators $r_{ij}^I$, $i,j\notin I$.
\end{theorem}
Note that for $k\geq n$, $\BPl{k}=0$ is trivial as no subset of size $k+1$ of $\underline{n}$ can be chosen. 
\begin{proof}
In order to prove the theorem, we have to show that there are no obstructions of degree larger than two (see e.g. \cites{Ani, CPU} for the terminology). We chose the following ordering on the generators $q_{ij}^I$. We first order by size of $(i,j)$ lexicographically. Given same subscripts, we order according to the lexicographic order on the superscripts $I$. The monomials are then ordered graded reverse lexicographically (\texttt{degrevlex} order). There are two different types of normal words of degree two:
\begin{align*}
r_{ij}^I r_{jl}^J&, & \text{for }i<j<J\setminus \lbrace J\cap I \cup i\rbrace <l,\\
r_{ij}^Ir_{jl}^I&, &j\notin I,
\end{align*}
with $i\notin I$, $i\in J$.

We claim that a basis for $\BPl{k}$ is given by monomials of the form
\begin{align}\label{monomialbasis}
r_{i_0,i_1}^{I_0}r_{i_1,i_2}^{I_1}\cdot \ldots \cdot r_{i_{t-1},i_t}^{I_{t-1}},
\end{align}
with $i_0<i_1<\ldots <i_t$, where for each $j=1,\ldots, t-1$ we have  for $r_{i_{j-1},i_{j}}^{I_{j-1}}r_{i_j,i_{j+1}}^{I_j}$ either 
\begin{align}\label{situation1}
&I_{j-1}=I_j,\quad\text{ or }\\\label{situation2}
&i_{j-1}\in I_j \quad \text{ and }\quad i_j<I_j\setminus\lbrace I_j\cap  I_{j-1}\cup i_{j-1}\rbrace<i_{j+1}.
\end{align}
To prove this, we note in an arbitrary non-zero monomial $M$ we might have degree two sub-word of the form $r_{i_{j-1}, i_j}^{I_{j-1}}r_{i_j,i_{j+1}}^{I_j}$ where $i_j$ is not necessarily smaller than all elements in $I_{j}\setminus \lbrace I_{j}\cap I_{j-1}\cup i_{j-1}\rbrace$. In this case, we can replace the sub-word by $r_{i_{j-1}, i_j'}^{I_{j-1}}r_{i_j',i_{j+1}}^{I_j}$, where $i_j'$ is smaller than all elements in $I_{j}\setminus \lbrace I_{j}\cap I_{j-1}\cup i_{j-1}\rbrace$ using relation (\ref{dualrelation3}). Assume that $j$ corresponds to right-most occurrence of such a degree two sub-word. If now $r_{i_{j-2}, i_{j-1}}^{I_{j-2}}r_{i_{j-1},i_{j}}^{I_{j-1}}$ is of the same form, but there exists an element of $I_{j-1}$ which is larger than $i_j'$, the monomial $M$ was zero by relation (\ref{dualrelation2}). Hence, such a situation cannot occur and by replacing all the non-normal degree two sub-words of $M$ we obtain that $M$ equals a monomial of the form (\ref{monomialbasis}).

It is now clear by the description of the monomial basis in (\ref{monomialbasis}) that the quadratic relations given in Lemma \ref{dualrelations} give a non-commutative Gr\"obner basis (non-commutative PBW basis) for the algebra $\BPl{k}$.
\end{proof}

We note, in particular, that $\BPl{k}$ is a monomial algebra if and only if $k=2$.

\begin{corollary}
The algebra $\BPl{k} $is Koszul, and hence the algebra $\RPl{k}$ is non-homogeneous Koszul for all  $k,n\geq 2$.
\end{corollary}

\begin{example}$~$\begin{enumerate}
\item[(i)]
Consider $\BPl{2}$ for small values of $n$. The algebra $B_3^{(2)}$  has a basis given by
$$1<r_{12}^3<r_{13}^2<r_{23}^1< r_{12}^3 r_{23}^1,$$
so the Hilbert series are
\begin{align}
H(B_3^{(2)},t)&=1+3t+t^2,\\
\begin{split}
H(R_3^{(2)},t)&=(1-3t+t^2)^{-1}\\&
=1+3t+8t^2+21t^3+55t^4+144t^5+O(t^6).
\end{split}
\end{align}
According to \cites{Ani2,Ani} (see \cite{CPU}*{Theorem 7.1}), this implies that the global dimension of $R_3^{(2)}$ equals two. The $n$-th coefficient of $H(R_3^{(2)},t)$ is the $2(n-1)$-th Fibonacci number.\footnote{According to the \emph{On-Line Encyclopaedia of Integer Sequence}\textregistered , \url{https://oeis.org/A001906}.}
\item[(ii)] The Hilbert series $k=2$ and $n=4$ is
\begin{align}
\begin{split}
H(R_4^{(2)},t)&=(1-12t+12t^2-5t^3)^{-1}\\&=1+12t+132t^2+1,445t^3+O(t^4).
\end{split}
\end{align}
\item[(iii)] The Hilbert series $k=2$ and $n=5$ is
\begin{align}\begin{split}
H(R_5^{(2)},t)&=(1-30t+50t^2-45t^3+17t^4)^{-1}
\\&=1 + 30 t + 850 t^2 + 24,045 t^3 + 680,183 t^4 +O(t^5).
\end{split}
\end{align}
\end{enumerate}
\end{example}

In general, the top degree of $H(\BPl{2},t)$ is $n-1$. The leading coefficient $h_{n-1}$ is given by
\begin{equation}\label{count1}
h_{n-1}=\sum_{i=3}^n (n-i+1)2^{n-i}=\sum_{i=0}^{n-3}(i+1)2^i,
\end{equation}
while the coefficient $h_1=\tfrac{1}{2}n(n-1)(n-2)$. The other coefficients can be computed as
\begin{equation}\label{count2}
h_{n-l}=\binom{n}{l-1}\left( l-1+\sum_{i=0}^{n-l-2} (i+l)2^{i}\right),
\end{equation}
for $1\leq l\leq n-1$. This can be seen by systematically counting normal words in the algebra $\BPl{2}$. Note that in top degree, these are of the form
\begin{align*}
r_{12}^{j_1}r_{23}^{j_2}\ldots r_{n-1,n}^{j_{n-1}},
\end{align*}
where for each $i=2,\ldots, n-1$ we can either have $j_i=j_{i-1}$ or $j_i=i-1$. In Equation (\ref{count1}) we count such monomials where $j_1=\ldots=j_{i-2}$ for $i=3,\ldots,n$ separately. Using the same counting method for an arbitrary ordered subset of size $n-l$ in $\underline{n}$,  Equation (\ref{count2}) follows.

\begin{example}
If $k=3$ and $n=4$, then $r_{12}^{\lbrace 3,4\rbrace}r_{24}^{\lbrace 1,3\rbrace}$ is the only non-zero quadratic monomial in $B_4^{(3)}$, and hence
\begin{align}
\begin{split}
H(R_4^{(3)},t)&=(1-6t+1t^2)^{-1}\\
&=1 + 6 t + 35 t^2 + 204 t^3 + 1,189 t^4 +  O(t^5),
\end{split}
\end{align}
for which the coefficients satisfy the recursion $a_n=6a_{n-1} - a_{n-2}$, with $a_0=1$.

In general, we find that $H(R_n^{(n-1)},t)=(1-\tfrac{1}{2}n(n-1)t+t^2)^{-1}$ since the monomial $r_{12}^{\underline{n}\setminus\lbrace 1,2\rbrace}r_{2n}^{{\underline{n}\setminus\lbrace 1,n\rbrace}}$ is the only non-zero quadratic monomial in $B_n^{(n-1)}$, and hence the coefficients of this Hilbert series satisfy the recursion $a_n=\tfrac{1}{2}n(n-1)a_{n-1} - a_{n-2}$.
\end{example}


\section{Koszulness of the Colimit Algebra \texorpdfstring{$R_n$}{Rn}}\label{limitalgebra}

The relations in \cite{GGRW}*{4.8.1} link the quasi-Pl\"ucker coordinates for $k\times n$-matrices with those of $(k-1)\times n$-matrices. In our algebraic setting, this gives the non-homogeneous relations
\begin{align}\label{extrarelation}
q_{ij}^J&=q_{ij}^{J\cup m}+ q_{im}^J\cdot q_{mj}^{J\cup i},
\end{align}
where $i,m\notin J$. This relation links $\QPl{k+1}$ and $\QPl{k}$, where $J$ is a set of size $k-1$. We can inductively define the quadratic algebra $\QPl{\leq k}$ as the coproduct of the algebras $\QPl{k'}$, for $k'\leq k$, with the additional relations of the form (\ref{extrarelation}). Accordingly, we define the \emph{algebra of quasi-Pl\"ucker coordinates} $Q_n$
\[
Q_n:=\QPl{\leq n}.
\]
Note that $\QPl{k}=0$ for $k\geq n$ as then it is not possible to choose an index set of size $k+1$ in $\underline{n}$. The colimit algebra $Q_n$ is again a quadratic-linear algebra with finitely many generators.

\begin{lemma}
The subalgebra $\RPl{\leq k}$ of $\QPl{\leq k}$ generated by $q_{ij}^J$ with $i<j$ can be described as the quotient of the colimit over the subalgebras $\RPl{k}$ together with the relations (\ref{extrarelation}) for $i<m<j$.
\end{lemma}
\begin{proof}
In the larger algebra $\QPl{\leq k}$, all relations of the form (\ref{extrarelation}) can be transformed into relations of the same form where the lower indices are in strictly increasing order. This can be checked distinguishing cases depending on the order of $\lbrace i,m,j\rbrace$ according to size, and multiplying by the correct inverse. Hence all the relations in the subalgebra $\RPl{\leq k}$ are of the same form.
\end{proof}

Therefore, we define the  colimit algebra
\[
R_n:=\RPl{\leq n}.
\]

\begin{theorem}
The algebras $\RPl{\leq k}$ are quadratic-linear Koszul algebras, and hence the qua\-drat\-ic-linear algebra $R_n$ is Koszul.
\end{theorem}
\begin{proof}
The quadratic part of the relation (\ref{extrarelation}) gives that
\begin{align}
q_{ij}^Jq_{jm}^{J\cup{i}}&=0, &\forall i<j<m,
\end{align}
where $i,j\notin J$. Consider the quadratic dual $\BPl{\leq k}$ of $(\RPl{\leq k})^{(0)}$. In this algebras, all products of generators $r_{ij}^Jr_{ab}^K$ with different sizes of the index sets $J,K$ are zero unless $j=a$ and $K=J\cup  i$. 
We extend the linear ordering on generators by requiring that $q_{ij}^I<q_{kl}^K$ if $|K|<|L|$, again using the \texttt{degrevlex} ordering on monomials.
Then a non-commutative PBW basis is given by products $M_{1}M_{2}\ldots M_{s}$ of monomials of the form from (\ref{monomialbasis}) which only give a non-zero product  if the last generator in $M_t$ is of the form $r_{ij}^K$, and the first generator of $M_{t+1}$ has the form $r_{jl}^{K\cup i}$. This shows that a quadratic non-commutative Gr\"obner basis exists for $\BPl{\leq k}$. In particular, $(\RPl{\leq k})^{(0)}$ is Koszul, and so $\RPl{\leq k}$ is non-homogeneous Koszul.
\end{proof}

\begin{example}
Consider the algebra $Q_4$. The quadratic dual  has the PBW basis
\begin{gather*}r_{12}^3<r_{12}^4<r_{13}^2<r_{13}^4<r_{14}^2<r_{14}^3<r_{23}^1<r_{23}^4<r_{24}^1<r_{24}^3<r_{34}^1<r_{34}^2
\\<r_{12}^{34}<r_{13}^{24}<r_{14}^{23}<r_{23}^{14}<r_{24}^{13}<r_{34}^{12}\\
<r_{12}^3r_{23}^1 <r_{12}^3r_{24}^1 <r_{12}^3r_{24}^3 <r_{12}^3r_{24}^{13}
 <r_{12}^4r_{23}^1 <r_{12}^4r_{23}^4 <r_{12}^4r_{24}^1 <r_{12}^4r_{23}^{14} \\ <r_{13}^2r_{34}^1<r_{13}^2r_{34}^2<r_{13}^2r_{34}^{12}<r_{13}^4r_{34}^1<r_{23}^1r_{34}^1<r_{23}^1r_{34}^2<r_{23}^1r_{34}^{12}<r_{23}^4r_{34}^2\\
<r_{12}^3r_{23}^1r_{34}^1<r_{12}^3r_{23}^1r_{34}^2<r_{12}^3r_{23}^1r_{34}^{12}<r_{12}^4r_{23}^1r_{34}^1\\<r_{12}^4r_{23}^1r_{34}^2<r_{12}^4r_{23}^1r_{34}^{12}<r_{12}^4r_{23}^4r_{34}^2.
\end{gather*}
Hence the Hilbert series for $Q_4$ is given by
\begin{equation}
\begin{split}
H(Q_4,t)&=(1-18t+16t^2-7t^3)^{-1}\\&=1 + 18 t + 308 t^2 + 5,263 t^3 + 89,932 t^4 +  O(t^5)
\end{split}
\end{equation}
\end{example}


\section{The Algebras \texorpdfstring{$\QPl{k}$}{Qn(k)} and \texorpdfstring{$Q_n$}{Qn} are also Koszul}\label{limitalgebrakoszul}

The non-homogeneous quadratic algebras $\QPl{k}$ can also be shown to be Koszul. However, it is not quadratic-linear, as constant terms appear in the relations (cf. \cite{PP}*{Chapter~5}). We change the presentation from Section \ref{Rnk} slightly:

\begin{lemma}
The algebra $\QPl{k}$ has generators $q_{ij}^I$, where $|I|=k-1$ and $i\notin I,$ subject to the relations
\begin{enumerate}
\item[(i)] $q_{ij}^I$ does not depend on the ordering of the elements of $I$;
\item[(ii)] $q_{ij}^I=0$ whenever $j\in I$;
\item[(iii)] $q_{ii}^I=1$, and $q_{ij}^Iq_{jl}^I=q_{il}^I$, $i,j\notin I$;
\item[(v')] If $i\notin M$, $i\in L$, then $\sum_{j\in L\setminus \lbrace i\rbrace} q_{ij}^Mq_{jl}^{L\setminus\lbrace j,l \rbrace}+q_{il}^{L\setminus\lbrace i,l \rbrace}=0$.
\end{enumerate}
\end{lemma}

\begin{theorem}\label{mainthm}
The non-homogeneous quadratic algebras $\QPl{k}$ and $\QPl{\leq k}$ (and hence, in particular, $Q_n$) are non-homogeneous Koszul.
\end{theorem}
\begin{proof}
The proof is similar to that for $\RPl{k}$ in Theorem~\ref{RKoszulTheorem}, but there are less restrictions of the order of indices.
We consider $\CPl{k}$, which is the quadratic dual of the associated quadratic algebra $(\QPl{k})^{(0)}$ from \cite{PP}*{Section 4.1}. Denote generators for this algebra by $r_{ij}^K$ (dual to $q_{ij}^K$). Then $r_{ij}^Kr_{ab}^L=0$ if $j\neq a$ or, if $j=a$, $r_{ij}^Kr_{jb}^L=0$ if $K\neq L$ and $i\notin L$. The relations in $\CPl{k}$ are fully described by
\begin{align}
r_{ij}^Mr_{jl}^{L\setminus \lbrace j,l\rbrace}&=r_{ij'}^Mr_{j'l}^{L\setminus \lbrace j',l\rbrace}, &\forall j,j'\in L\setminus\lbrace L\cap M\cup i\rbrace,
\end{align}
requiring distinct sub-indices and $i\in L$.
This means we have a rewriting rule
\begin{align*}
r_{ij}^Mr_{jl}^{L}&\mapsto r_{il_{L\setminus M}}^Mr_{l_{L\setminus M}l}^{L'},
\end{align*}
where $l_{L\setminus M}=\min L\setminus (L\cap M\cup i)$, and $L'=(L\setminus l_{L\setminus M}  )\cup j$.
One checks that for a critical triple $r_{ij}^Ir_{jl}^Jr_{lm}^L$, applying the rewriting rule to the first two generators and then to the last two generators gives a reduced monomial, and the same reduced monomial emerges if we apply the rewriting rules in opposite order. Hence, the process of applying rewriting rules stabilizes after two steps. This means every critical pair is confluent, and hence the algebra is Koszul (cf. e.g. \cite{LV}*{Section 4.1} for these general results and terminology). This implies that $\QPl{k}$ is non-homogeneous Koszul. Further, after adding relation (\ref{extrarelation}), the same is true. Moreover, the process of passing to the quadratic part of relations still commutes with taking the coproduct, and hence the algebras $\QPl{\leq k}$ are also non-homogeneous Koszul.
\end{proof}

A consequence of Theorem \ref{mainthm} is that it gives a non-homogeneous PBW basis for $\QPl{k}$, cf. \cite{PP}*{Sections 4.4, 5.2}.  Note that an alternative approach to finding a basis for an algebra of quasi-Pl\"ucker coordinates was given in \cite{Lau2}*{Section 7.4}.


\section{Differential Gradings on the Quadratic Duals}\label{diffsect}

Using the non-homogeneous quadratic duality of \cite{PP}*{5.4}, it follows that the algebras $\BPl{k}$ are \emph{differentially graded (dg) algebras}. That is, for each of these algebras, there exist a graded map $\der$ of degree one such that
\begin{align}
\der^2=0, && \der(xy)=\der(x)y+(-1)^{\deg x} x \der(y),
\end{align}
for homogeneous $x$, which is referred to as the \emph{differential}. We study the case $k=2$ in more detail and relate it to certain refinements of triangles with labelled corners.

To a generator $r_{ij}^k$ with $i<j$, $k\neq i,j$ associate the  triangle with labelled corners
$$r_{ij}^k\quad\longleftrightarrow\quad \vcenter
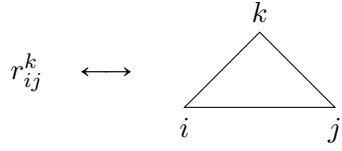
{\hbox{\begin{tikzpicture}
\draw (0,0) node[anchor=north]{$i$}
  -- (2,0) node[anchor=north]{$j$}
  -- (1,1) node[anchor=south]{$k$}
  -- cycle;
\end{tikzpicture}}}$$
Consider three types of ways to add a corner to the triangles:
\begin{align}
\vcenter{\hbox{
\begin{tikzpicture}
\draw (0,0) node[anchor=north]{$i$}
  -- (2,0) node[anchor=north]{$j$}
  -- (1,1) node[anchor=south]{$k$}
  -- cycle;
\end{tikzpicture}}}&&\longrightarrow&& -\vcenter{\hbox{
\begin{tikzpicture}
\draw (0,0) node[anchor=north]{$i$}  --(1,1)
  -- (2,0) node[anchor=north]{$j$}
    -- (1,0) node[anchor=north]{$l$}
  -- (1,1) node[anchor=south]{$k$} --(2,0)
  -- cycle;
\end{tikzpicture}}}, && i<l<j, \text{ and }l\neq k,\label{triangulation1}
\end{align}
\begin{align}
\vcenter{\hbox{
\begin{tikzpicture}
\draw (0,0) node[anchor=north]{$i$}
  -- (2,0) node[anchor=north]{$j$}
  -- (1,1) node[anchor=south]{$k$}
  -- cycle;
\end{tikzpicture}}}&&\longrightarrow&& -\vcenter{\hbox{
\begin{tikzpicture}
\draw (0,0) node[anchor=north]{$i$}  
  -- (2,0) node[anchor=north]{$j$} 
   -- (1.5,0.5) node[anchor=south]{$l$}
  -- (1,1) node[anchor=south]{$k$} -- (0,0)-- (1.5,0.5);
\end{tikzpicture}}}, && i<l<j, \text{ and }l\neq k,\label{triangulation2}
\end{align}
\begin{align}
\vcenter{\hbox{
\begin{tikzpicture}
\draw (0,0) node[anchor=north]{$i$}
  -- (2,0) node[anchor=north]{$j$}
  -- (1,1) node[anchor=south]{$k$}
  -- cycle;
\end{tikzpicture}}}&&\longrightarrow&& \vcenter{\hbox{
\begin{tikzpicture}
\draw (0,0) node[anchor=north]{$i$}  
  -- (2,0) node[anchor=north]{$j$} 
   -- (1.5,0.5) node[anchor=south]{$k$}
  -- (1,1) node[anchor=south]{$l$} -- (0,0)-- (1.5,0.5);
\end{tikzpicture}}}, && i<k<j, \text{ and }l\neq i,k.\label{triangulation3}
\end{align}
The triangulations on the right hand side correspond to the products 
\begin{align*}
-r_{il}^kr_{lj}^k &&\longleftrightarrow && (\ref{triangulation1}), && -r_{il}^kr_{lj}^i &&\longleftrightarrow &&(\ref{triangulation2}),&&r_{ik}^lr_{kj}^i &&\longleftrightarrow && (\ref{triangulation3}).
\end{align*}
We can recover a quadratic monomial from a triangulation by reading from left to right, reflecting the second triangle in the cases (\ref{triangulation2}) and (\ref{triangulation3}) so that the left corner becomes the top corner.

Now the map $\der(r_{ij}^k)$ is the sum over all ways to triangulate $\vcenter{\hbox{\begin{tikzpicture}
\draw (0,0) node[anchor=north]{$i$}
  -- (1,0) node[anchor=north]{$j$}
  -- (0.5,0.5) node[anchor=south]{$k$}
  -- cycle;
\end{tikzpicture}}}$ by adding one corner in any of the three ways described in (\ref{triangulation1})--(\ref{triangulation3}).

\begin{corollary}$~$ 
The map $\der_B\colon \BPl{2}\to \BPl{2}$ given by
\begin{align}
\der_B(r_{ij}^k)&=\begin{cases}
-\sum_{l=i+1}^{j-1}(r_{il}^kr_{lj}^k+r_{il}^kr_{lj}^i)+\sum_{l\neq i}r_{ik}^lr_{kj}^i,&i<k<j,\\
-\sum_{l=i+1}^{j-1}(r_{il}^kr_{lj}^k+r_{il}^kr_{lj}^i), & \text{otherwise}.
\end{cases}
\end{align}
is a differential for $\BPl{2}$.
\end{corollary}

We can explicitly compute the homology of these dg algebras in small examples:  the homology of $B_3^{(2)}$ has graded dimensions $1+2t$, and the homology of $B_4^{(2)}$ has graded dimensions $1+7t+2t^2$. 

The algebras $C_n^{(k)}$ will not give dg algebras, but rather examples of non-trivial \emph{curved} dg algebras \cite{PP}*{5.4, Definition~1}.

\section{Non-commutative Flag Coordinates}\label{flagsect}

Note that in addition to the algebra of quasi-Pl\"ucker coordinates, one can consider the algebra of \emph{flag coordinates}. Flag coordinates also generalize to non-commutative entries, using quasi-determinants \cite{GR4}*{Section II.2.7}, \cite{GGRW}*{Section 4.10}. Given a $k\times n$-matrix, $n\geq k$, choose distinct indices $i,j_1,\ldots j_{k-1}$ in $\underline{n}$ and denote
\begin{equation}
f_{i,\lbrace j_1,\ldots, j_k\rbrace}(A) =\begin{vmatrix}
a_{1i} &a_{1j_1}&\ldots& a_{1j_{k-1}}\\
\vdots &\vdots&\ddots &\vdots\\ 
a_{ki}&a_{kj_1}& \ldots &a_{kj_{k-1}}\end{vmatrix}_{k1},
\end{equation}
which is independent of the order of $\lbrace  j_1,\ldots, j_k\rbrace$. These functions are referred to as \emph{non-commutative flag coordinates}  and were introduced in \cite{GR1}.

For a set $I$ of smaller size, one can consider $f_{i,I}(A)$ by restricting to the first $|I|+1$ rows of $A$. Then the following relations hold \cite{GGRW}*{Section 4.10.2}:
\begin{align}
f_{i, I}f_{i, I\setminus k}^{-1}&=-f_{k, I\setminus k\cup i}f_{k, I\setminus k}^{-1}& \forall k\in I,\label{flagrel1}\\
\sum_{i=1}^kf_{j_i,J\setminus j_i}f_{j_i,J\setminus \lbrace j_i,j_{i-1}\rbrace}^{-1}&=0,\label{flagrel2}
\end{align}
where $J=\lbrace j_1,\ldots, j_k\rbrace$ and we also denote $j_0=j_k$.

\begin{definition}
For $n\geq 2$, we denote by $\FPl{n}$  the non-homogeneous quadratic algebra with generators $f_{i,I}$, where $I$ is a subset of $\underline{n}$ and $i\notin I$, and relations given by (\ref{flagrel1})--(\ref{flagrel2}) as well as \begin{equation}\label{flagrel3}
f_{i,I}f_{i,I}^{-1}=f_{i,I}^{-1}f_{i,I}=1.
\end{equation}
\end{definition}

Note that by virtue of the relations
\begin{equation}
q_{ij}^I(A)=f_{i,I}(A)^{-1}f_{j,I}(A),
\end{equation}
there exists a homomorphism of algebras from $\QPl{k}$ to the quotient skew-field of $F_n$ \cite{GGRW}*{Section 4.10}.
See also \cite{Lau2}*{Proposition 70}.

\begin{theorem}
The algebras $\FPl{n}$ are non-homogeneous Koszul.
\end{theorem}
\begin{proof}
Consider the quadratic dual $G_n:=(F_n^{(0)})^!$ of the homogeneous part of the relations (\ref{flagrel1})--(\ref{flagrel3}). In this algebra, denoting the dual generator for $f_{i,I}$ by $g_{i,I}$, we have
\begin{align}
g_{i,I}g_{i,I}^{-1}&\neq 0,\qquad g_{i,I}^{-1}g_{i,I}\neq 0,\\
g_{i, I}g_{i, I\setminus k}^{-1}&=g_{k, I\setminus k\cup i}g_{k, I\setminus i}^{-1}& \forall k\in I,\\
g_{j,J\setminus j}g_{j,J\setminus \lbrace j,k\rbrace}^{-1}&=g_{l,J\setminus l}g_{l,J\setminus \lbrace l,j\rbrace}^{-1},& l,j\in J,\
\end{align}
plus all other quadratic monomials not appearing in these relations are zero.
We order the generators $g_{i,I}^{\pm 1}$ lexicographically according to the triple $(|I|,i,I)$ and $g<g^{-1}$, and use the \texttt{degrevlex} ordering on monomials. Then the normal words of degree two are
\begin{align*}
g_{i,I}g_{i,I}^{-1},&&g_{i,I}^{-1}g_{i,I},&&g_{j,J}g_{j,J\setminus k}^{-1},
\end{align*}
where $i\notin I$, and $j<J$ (in particular $j<k$, otherwise no restriction on $k\in J$).
The rewriting rule is given by
\begin{align*}
g_{i, I}g_{i, I\setminus k}^{-1}\mapsto g_{m_I,I\setminus m_I\cup i}g_{m_I,I\setminus m_I}^{-1},
\end{align*}
where $m_I=\min I$.

Monomials in which every two neighboring generators are one of these normal words give a basis for $G_n$, which is thus a non-commutative PBW basis, and hence $F_n$ is non-homogeneous Koszul.
\end{proof}

\bibliography{biblio}

\bibliographystyle{amsrefs}

\end{document}